\documentclass[12pt]{amsproc}

\usepackage{amsmath, amssymb, amsfonts, amsthm, verbatim}
\usepackage{setspace}
\usepackage{latexsym}
\usepackage{float}
\usepackage{color}
\usepackage{mathrsfs}
\usepackage{hyperref}      
\usepackage{bm}
\usepackage[margin=1in]{geometry}
\usepackage{enumerate}
\usepackage{url}
\usepackage{hyperref} 
\usepackage{colonequals}

\numberwithin{equation}{section}
\newtheorem{thm}{Theorem}[section]
\newtheorem{cor}[thm]{Corollary}

\newtheorem{lem}[thm]{Lemma}

\newtheorem{prop}[thm]{Proposition}
\theoremstyle{definition}\newtheorem{defn}[thm]{Definition}
\theoremstyle{definition}\newtheorem{rem}[thm]{Remark}
\theoremstyle{definition}\newtheorem{ex}[thm]{Example}

\newcommand{\GF}{{\rm GF}}

\newcommand{\PSL}{{\rm PSL}}

\newcommand{\GL}{{\rm GL}}

\newcommand{\PGL}{{\rm PGL}}
\newcommand{\AGL}{{\rm AGL}}

\newcommand{\Sym}{{\rm Sym}}

\newcommand{\liso}{\lesssim}
\newcommand\Wr{\mathrel{\rm wr }}

\renewcommand\le{\leqslant}
\renewcommand\ge{\geqslant}

\renewcommand{\exp}{{\rm exp}}

\newcommand{\Z}{\mathbb{Z}}

\newcommand{\Ker}{{\rm Ker}}

\newcommand{\wexp}{weakly exponential}
\newcommand{\wic}{wexp-solvable}
\newcommand{\Wic}{Weakly exponential-solvable}
\newcommand{\indexc}{wexp-nonsolvable}
\newcommand{\IEe}{minimal wexp-nonsolvable}

\newcommand{\exptriv}{exp-trivial}
\newcommand{\expsimp}{exp-simple}

\newcommand{\lexp}{\le_{exp}}
\newcommand{\lwexp}{\le_{wexp}}

\begin{document}

\title{Exponential and weakly exponential subgroups of finite groups}

\author{Eric Swartz}
\address{Department of Mathematics, William \& Mary, Williamsburg, VA 23187}
\email{easwartz@wm.edu}

\author{Nicholas J. Werner}
\address{Department of Mathematics, Computer and Information Science, SUNY at Old Westbury, Old Westbury, NY 11568}
\email{wernern@oldwestbury.edu}

\begin{abstract}
Sabatini \cite{Sabatini_2024} defined a subgroup $H$ of $G$ to be an \textit{exponential subgroup} if $x^{|G:H|} \in H$ for all $x \in G$, in which case we write $H \lexp G$.  Exponential subgroups are a generalization of normal (and subnormal) subgroups: all subnormal subgroups are exponential, but not conversely.  Sabatini proved that all subgroups of a finite group $G$ are exponential if and only if $G$ is nilpotent.  The purpose of this paper is to explore what the analogues of a simple group and a solvable group should be in relation to exponential subgroups.  We say that an exponential subgroup $H \lexp G$ is \textit{\exptriv{}} if either $H = G$ or the exponent of $G$, $\exp(G)$, divides $|G:H|$, and we say that a group $G$ is \textit{\expsimp{}} if all exponential subgroups of $G$ are \exptriv.  We classify finite \expsimp{} groups by proving $G$ is \expsimp{} if and only if $\exp(G) = \exp(G/N)$ for all proper normal subgroups $N$ of $G$, and we illustrate how the class of \expsimp{} groups differs from the class of simple groups.  Furthermore, in an attempt to overcome the obstacle that prevents all subgroups of a generic solvable group from being exponential, we say that a subgroup $H$ of $G$ is \textit{weakly exponential} if, for all $x \in G$, there exists $g \in G$ such that $x^{|G:H|} \in H^g$.  If all subgroups of $G$ are weakly exponential, then $G$ is \textit{\wic{}}.  We prove that all solvable groups are \wic{} and almost all symmetric and alternating groups are not \wic{}.  Finally, we completely classify the groups $\PSL(2,q)$ that are \wic{}. We show that if $\pi(n)$ denotes the number of primes less than $n$ and $w(n)$ denotes the number of primes $p$ less than $n$ such that $\PSL(2,p)$ is \wic,   then
 \[ \lim_{n \to \infty} \frac{w(n)}{\pi(n)} = \frac{1}{4}.\]
\end{abstract}

\maketitle

\section{Introduction}

In this paper, all groups are finite. It is a basic property of a subnormal subgroup $H\mathrel{\lhd\lhd}G$ that $x^{|G:H|} \in H$ for all $x \in G$. Subgroups that are not (sub)normal may or may not exhibit this behavior. For instance, the dihedral group $D_4$ of order 8 contains a nonnormal subgroup $H$ of index 4, but $x^4 \in H$ for all $x \in D_4$. On the other hand, if $x$ and $y$ are distinct transpositions in the symmetric group $S_3$, then $x^3 \notin \langle y \rangle$.

\begin{defn}
 Let $G$ be a finite group.  A subgroup $H \le G$ is said to \textit{exponential} with respect to $G$ if, for all $x \in G$, $x^{|G:H|} \in H$. In this case, following \cite{Sabatini_2024}, we write $H \lexp G$. 
\end{defn}

It was shown in \cite[Proposition 5.5]{Sabatini_2024} (see also \cite{Werner_2024}) that, for a finite group $G$, $H \lexp G$ for all subgroups $H \le G$ if and only if $G$ is nilpotent. This prompts the question of whether the exponential subgroups, or variations thereof, can be used to identify other classes of finite groups. Our purpose in this paper is to investigate this problem, with an emphasis on simple groups and solvable groups.

As noted above, the exponential property is a weak form of subnormality: if $N\mathrel{\lhd \lhd}G$, then $N$ is exponential with respect to $G$, but not conversely.  The question then becomes: how much weaker than normality is the exponential property?  Just as every group has trivial normal subgroups, there is a class of subgroups that trivially are exponential; for example, $G \lexp G$. Denote the exponent of $G$ by $\exp(G)$. Then, it is clear that $H$ is exponential with respect to $G$ whenever $\exp(G)$ divides $|G:H|$, and that $G$ always contains such subgroups. As a simple group is one that contains no nontrivial normal subgroups, a group that is ``simple'' with respect to the exponential property should contain as few exponential subgroups as possible.

\begin{defn}
 Let $H \le G$.  We say that $H$ is \textit{exponential-trivial} (or \textit{\exptriv{}}) if either $H = G$ or $\exp(G)$ divides $|G:H|$. A group $G$ is said to be \textit{exponential-simple} (or \textit{\expsimp{}}) if the only exponential subgroups of $G$ \exptriv{}.
\end{defn}

Which groups are \expsimp{}?  How does the class of simple groups compare to the class of \expsimp{} groups?  We are able to characterize \expsimp{} groups and answer these questions precisely.

\begin{thm}
 \label{thm:charIIPsimple}
 Let $G$ be a finite group.  Then, $G$ is \expsimp{} if and only if $\exp(G) = \exp(G/N)$ for all proper normal subgroups $N$ of $G$.
\end{thm}

Using this, we can prove the following:

\begin{cor}
 \label{cor:IIPsimple}
 Let $G$ be a finite group.
 \begin{enumerate}[(1)]
  \item If $G$ is simple, then $G$ is \expsimp{}.
  \item If $G$ is solvable, then $G$ is \expsimp{} if and only if $G$ is a $p$-group with exponent $p$.
  \item If $G$ is the direct product of simple groups, then $G$ is \expsimp{} if and only if each simple direct factor of $G$ has the same exponent.
  \item If $G$ is quasisimple, then $G$ is \expsimp{} if and only if $\exp(G) = \exp(G/Z(G))$.
 \end{enumerate}
\end{cor}

See Section \ref{sect:IIPsimple} for the proofs of Theorem \ref{thm:charIIPsimple} and its corollary.

%\begin{proof}
% See Corollaries, \ref{cor:simpleIIPsimple}, \ref{cor:solvIIPsimple}, \ref{cor:directprod_IIPsimple}, and \ref{cor:quasiIIP}.
%\end{proof}

Recall that all subgroups $H$ of $G$ are exponential if and only if $G$ is nilpotent.  We would like to modify the definition of exponential subgroups in such a way as to capture all solvable groups.  Consider the following example.

\begin{ex}
\label{ex:easy}
 Let $G = S_3$.  This group will have subgroups that are not exponential since $G$ has more than one Sylow $2$-subgroup. Indeed, take $x = (1 \; 2)$ and $H = \langle (2 \; 3) \rangle$.  Then, $|G:H| = 3$, and hence $x^{|G:H|} = x \notin H$, showing that $H$ is not exponential.  We will run into such a problem whenever a group $G$ has more than one Sylow $p$-subgroup. Suppose $P_1 \neq P_2$ are such subgroups and $x \in P_1 \backslash P_2$.  Then, $|G:P_2|$ is coprime to $p$, so $\langle x^{|G:P_2|} \rangle = \langle x \rangle$. Thus, $x^{|G:P_2|} \notin P_2$.
\end{ex}

To avoid the obstruction discussed in Example \ref{ex:easy}, we can use the fact that all Sylow $p$-subgroups of a finite group are conjugate. This leads to the following natural generalization of exponential subgroups.

\begin{defn}
 Let $G$ be a finite group.  A subgroup $H \le G$ is said to be \textit{\wexp{}} with respect to $G$ if, for all $x \in G$, there exists $g \in G$ such that $x^{|G:H|} \in H^g$.  In this case, we write $H \lwexp G$.
\end{defn} 
 
%  If every subgroup of $G$ is \wexp, then $G$ is called a \textit{\wic{} group}.  If $G$ is not \wic{}, then $G$ is said to be an \textit{\indexc{} group.} 
% \end{defn}

It should be immediately observed that all subgroups of $S_3$ are \wexp{}, and, indeed, if $P$ is a Sylow $p$-subgroup of $G$, we will have $P \lwexp G$, fixing the issues raised in Example \ref{ex:easy}.  With this in mind, we make the following definition.

\begin{defn}
 If every subgroup of $G$ is \wexp, then $G$ is called a \textit{weakly exponential-solvable group} (or a \textit{\wic{} group}).  If $G$ is not \wic{}, then $G$ is said to be a \textit{weakly exponential-nonsolvable group} (or \textit{\indexc{} group}).
\end{defn}

As the next theorem shows, the use of the term ``solvable'' in ``\wic'' is justified.
 
%We can now ask the question: how does the class of index closed groups compare to the class of solvable groups?  In fact, every solvable group is weakly index closed.
 
\begin{thm}
 \label{thm:solvwic}
 All finite solvable groups are \wic{} groups.
\end{thm}

It turns out that the class of \wic{} groups is strictly larger than the class of solvable groups (and hence a group being \wic{} is a weaker condition than being solvable). In particular, the alternating groups $A_5$ and $A_6$ are \wic; see Lemmas \ref{lem:A5} and \ref{lem:A6}. This also means that a group being \indexc{} is a strictly stronger condition than being nonsolvable.  It is not difficult to show (see Lemma \ref{lem:wicquo}) that if any proper quotient of a finite group $G$ is \indexc, then $G$ is itself \indexc.  This motivates the next definition, which is analogous to that of \textit{$\sigma$-elementary} when studying covering numbers of groups (see \cite{Garonzi_Kappe_Swartz_2022}).

\begin{defn}
 A finite group $G$ is said to be \textit{\IEe} if $G$ is \indexc{} (that is, $G$ is not \wic) but all proper quotients of $G$ are \wic.
\end{defn}

Thus, determining which finite groups are \wic{} (or, equivalently, which finite groups are \indexc) comes down to determining which groups are \IEe. It would be quite interesting to classify the \IEe{} groups. Any almost simple group is potentially \IEe, and natural examples to consider are symmetric and alternating groups. Apart from some small examples, these groups are all \IEe.

\begin{thm}
 \label{thm:Sn}
 The symmetric group $S_n$ is \IEe{} if and only if $n = 5$ or $n \ge 7$.
\end{thm}

\begin{thm}
 \label{thm:An}
 The alternating group $A_n$ is \IEe{} if and only if $n \ge 7$.
\end{thm}

These theorems might lead one to conjecture that almost all (almost) simple groups are \IEe.  However, the reality is far more complicated, as evidenced by the following result on projective special linear groups.

\begin{thm}
\label{thm:PSLwic}
Let $q = p^d$ be a prime power.  Then, $\PSL(2,q)$ is a \wic{} group if and only if $q = 4, 9$ or 
\[ p \equiv 2, 3, 5, 7, 17, 43, 53, 67, 77, 103, 113 \pmod {120}\]
and $d$ is odd.  Consequently, $\PSL(2,q)$ is \wic{} for infinitely many values of $q$ and is \indexc{} (and hence \IEe{}) for infinitely many values of $q$
\end{thm}
% \begin{enumerate}[(1)]
%  \item Let $p$ be a prime.  Then, $\PSL(2,p)$ is a \wic{} group if and only if $p = 2, 3, 5$ or 
% %Let $p > 3$ be an odd prime.  Then, $\PSL(2,p)$ is a weakly index closed group if and only if $p = 5$ or 
%  \[ p \equiv 7, 17, 43, 53, 67, 77, 103, 113 \pmod {120}.\]
%  Consequently, $\PSL(2,p)$ is \wic{} for infinitely many values of $p$, and $\PSL(2,p)$ is \indexc{} (and hence \IEe) for infinitely many values of $p$.
%  \item Let $p$ be a prime, $q = p^d, d \ge 2$.  Then, $\PSL(2,q)$ is \wic{} if and only if $q = 4, 9$ or $d$ is odd and $\PSL(2,p)$ is \wic.
%  
%  \end{enumerate}

Other than the three exceptional cases $p = 2, 3, 5$, by Theorem \ref{thm:PSLwic}, $\PSL(2,p)$ is \wic{} if and only if $p$ lies in one of eight congruence classes modulo $120$.  If $\varphi$ denotes the Euler totient function, then the prime numbers are asymptotically equally distributed among the $\varphi(120) = 32$ congruence classes modulo $120$ (this is the ``stronger version'' of Dirichlet's Theorem on Arithmetic Progressions).  This gives us the last of our major results.

\begin{cor}
 \label{cor:proportionPSLwic}
 Let $\pi(n)$ denote the number of primes less than $n$ and $w(n)$ denote the number of primes $p$ less than $n$ such that $\PSL(2,p)$ is \wic.  Then,
 \[ \lim_{n \to \infty} \frac{w(n)}{\pi(n)} = \frac{1}{4}.\]
\end{cor}

\begin{rem}
 Computation in GAP \cite{GAP4} shows that the only group with order at most $1000$ that is \indexc{} and does not have $S_5$ as a proper quotient is $\PSL(2,11)$.  In particular, the only \IEe{} groups with order at most $1000$ are $S_5$ and $\PSL(2,11)$.
\end{rem}

This paper is organized as follows.  Section \ref{sect:IIPsimple} is dedicated to the study of \expsimp{} groups and the proofs of Theorem \ref{thm:charIIPsimple} and Corollary \ref{cor:IIPsimple}.  In Section \ref{sect:prelims}, we present some basic results that are useful when studying \wexp{} subgroups.  Section \ref{sect:solv} is dedicated to the proof of Theorem \ref{thm:solvwic}, which shows that all solvable groups are \wic.  In Section \ref{sect:SnAn}, we prove Theorems \ref{thm:Sn} and \ref{thm:An}, which determines precisely which symmetric and alternating groups are \IEe.  Finally, in Section \ref{sect:PSL}, we prove Theorem \ref{thm:PSLwic}, which classifies the prime powers $q$ such that $\PSL(2,q)$ is \wic. % IE-elementary.

%%%%%%%%%%%%%%%%%%%%%%%%%%%%%%%%%%%%%%%%%%%%%%%%%
%%%%%%%%%%%%%%%%%%%%%%%%%%%%%%%%%%%%%%%%%%%%%%%%%
%%%%%%%%%%%%%%%%%%%%%%%%%%%%%%%%%%%%%%%%%%%%%%%%%

\section{Exponential-simple groups}
\label{sect:IIPsimple}

Recall that a group $G$ is \expsimp{} if the only exponential subgroups of $G$ are $G$ itself, and those proper subgroups $H \le G$ such that $\exp(G)$ divides $|G:H|$. As pointed out in \cite[Remark 4.2]{Sabatini_2024}, $G$ will contain exponential subgroups $H \ne \{1\}$ as long as $G$ is not a cyclic group of prime order. So, one should expect that there exist nontrivial examples of \expsimp{} groups, and this is indeed the case. In this section, we will prove Theorem \ref{thm:charIIPsimple}---which completely describes \expsimp{} groups---as well as its corollaries, which give several interesting examples of such groups.

% 
% and its corollaries, which completely describe \expsimp{} groups. We will also present several interesting examples of such groups.

One part of our classification of \expsimp{} groups will follow from the next lemma; see also \cite[Lemma 4.3]{Sabatini_2024}.

\begin{lem}
 \label{lem:core}
 Let $G$ be a finite group, $H \le G$, and 
 \[ K \colonequals \bigcap_{g \in G} H^g.\]
 If $H$ is exponential with respect to $G$, then, for all $x \in G$, $x^{|G:H|} \in K$.
 In particular, $\exp(G/K)$ divides $|G:H|$.
\end{lem}

\begin{proof}
 Suppose $H$ is exponential and let $x \in G$.  For all $g \in G$, we have
 \[ ( x^{|G:H|})^{g^{-1}} = (x^{g^{-1}})^{|G:H|} \in H.\]
Thus, $x^{|G:H|} \in H^g$ for all $g \in G$, and so $x^{|G:H|} \in K$. Finally, $K \lhd G$, so $\overline{x}^{|G:H|} = \overline{1}$ for all $\overline{x} \in G/K$, proving the result.
\end{proof}

We can now prove Theorem \ref{thm:charIIPsimple}, which characterizes \expsimp{} groups.

\begin{proof}[Proof of Theorem \ref{thm:charIIPsimple}]
 First, let $G$ be a finite \expsimp{} group and let $N$ be a proper normal subgroup of $G$.  Certainly, $\exp(G/N)$ divides $\exp(G)$.  Since $N$ is a normal subgroup, $N \lexp G$; but, $G$ is \expsimp{}, so $N$ is \exptriv{}, and hence $\exp(G)$ divides $|G:N|$.  
 
 Let $\overline{G} = G/N$, and suppose $\exp(G) > \exp(G/N)$.  Then, there exists a prime $p$ and positive integer $r$ such that $p^r$ divides $\exp(G)$ but $p^r$ does not divide $\exp(\overline{G})$.  Now, $p^r$ divides $\exp(G)$, so $p^r$ divides $|G:N| = |\overline{G}|$.  Suppose a Sylow $p$-subgroup of $\overline{G}$ has order $p^d$ and let $\overline{H}$ be a subgroup of $\overline{G}$ of order $p^{d - (r - 1)}$.  Let $\phi: G \to \overline{G}$ be the natural homomorphism and $H \colonequals \phi^{-1}(\overline{H}) \le G$.  Then,
 \[ |G:H| = |\overline{G}:\overline{H}| = p^{r-1} \cdot |\overline{G}|_{p'},\]
 where $|\overline{G}|_{p'} \colonequals |\overline{G}|/p^d$ denotes the $p'$-part of $|\overline{G}|$.  Since the highest power of $p$ dividing $\exp(\overline{G})$ is less than $p^r$, $\exp(\overline{G})$ divides $|G:H|$, but $p^r$ does not divide $|G:H|$, so $\exp(G)$ does not divide $|G:H|$.  
 
 Let $x \in G$.  In $\overline{G}$, $\overline{x}^{|G:H|} = \overline{1}$, so $x^{|G:H|} \in N \le H$.  Thus, $H \lexp G$, but $\exp(G) \nmid |G:H|$, so $H$ is not \exptriv{}. This contradicts $G$ being \expsimp{}.  Thus, if $G$ is \expsimp{}, then $\exp(G/N) = \exp(G)$ for all proper normal subgroups $N$.
 
Conversely, assume that $\exp(G) = \exp(G/N)$ for all proper normal subgroups $N$ of $G$.  Let $H < G$, and assume $H \lexp G$. Let
\begin{equation*}
K \colonequals \bigcap_{g \in G}\limits H^g.% \lhd G,
\end{equation*}
By Lemma \ref{lem:core}, $\exp(G/K)$ divides $|G:H|$.  By hypothesis, $\exp(G/K) = \exp(G)$, meaning $H$ is \exptriv{}.  Therefore, $G$ is \expsimp{}, as desired.
\end{proof}

%We have a number of immediate corollaries.

\begin{cor}
 \label{cor:simpleIIPsimple}
 If $G$ is a finite simple group, then $G$ is \expsimp{}.
\end{cor}

\begin{proof}
This is clear, since in a finite simple group the only proper normal subgroup is the identity subgroup. (This result also follows from \cite[Lemma 4.3]{Sabatini_2024}.)% If $G$ is a finite simple group $G$, then its only proper normal subgroup is the identity subgroup, and hence $\exp(G/N) = \exp(G)$ for all proper normal subgroups $N$ of $G$.  By Theorem \ref{thm:charIIPsimple}, $G$ is IIP-simple.
\end{proof}

Although all simple groups are \expsimp{}, the class of \expsimp{} groups is strictly larger than the class of simple groups, as the following results show.

\begin{cor}
 \label{cor:solvIIPsimple}
 Let $G$ be a solvable group.  Then, $G$ is \expsimp{} if and only if $G$ is a $p$-group with exponent $p$ for some prime $p$.
\end{cor}

\begin{proof}
 Let $G$ be a solvable group.  Assume first that $G$ is a $p$-group with exponent $p$.  Then, the exponent of every proper quotient also has exponent $p$, and so $G$ is \expsimp{} by Theorem \ref{thm:charIIPsimple}.  
 
 Conversely, assume that $G$ is a solvable group that is \expsimp{}.  Since $G$ is solvable, there exists a normal subgroup $N$ of $G$ such that $|G:N| = p$ for some prime $p$.  By Theorem \ref{thm:charIIPsimple}, this implies $\exp(G) = \exp(G/N) = p$, and hence $G$ is a $p$-group with exponent $p$.
\end{proof}

\begin{cor}
 \label{cor:directprod_IIPsimple}
 Let $T_1, \dots, T_n $ be simple groups.  The direct product 
 \[ G \colonequals T_1 \times \cdots \times T_n \]
 is \expsimp{} if and only if all $T_i$ have the same exponent.
\end{cor}

\begin{proof}
 Assume first that $G$ is \expsimp{}.  Since $G$ naturally projects onto each $T_i$, by Theorem \ref{thm:charIIPsimple}, this implies $\exp(T_i) = \exp(G)$ for all $i$. Conversely, assume each $T_i$ has the same exponent $m$.  If $N \lhd G$, then $G/N$ is isomorphic to a direct product of some collection of $T_i$'s, and so $\exp(G/N) = m = \exp(G)$.  By Theorem \ref{thm:charIIPsimple}, $G$ is \expsimp{}.
\end{proof}

\begin{ex}
 The groups $\PSL(4,2) \cong A_8$ and $\PSL(3,4)$ are not isomorphic, but both groups have exponent $420$.  Thus, $\PSL(4,2) \times \PSL(3,4)$ is \expsimp{} by Theorem \ref{thm:charIIPsimple}.
 \end{ex}
 
 \begin{ex}
Let $\exp_p(G)$ denote the highest power of $p$ dividing $\exp(G)$. If $p$ is an odd prime and $n=p^k$, then $A_n$ contains an element of order $p^k$, while $A_{n-1}$ does not. So, if $n$ is a power of $p$, then $\exp_p(A_n) \ne \exp_p(A_{n-1})$. Otherwise, $\exp_p(A_n) = \exp_p(A_{n-1})$. When $p=2$, elements of order $2^k$ are present in $A_n$ if and only if $n \ge 2^k+2$. Hence, if $n \ne 2^k+2$, then $\exp_2(A_n) = \exp_2(A_{n-1})$. Thus, if $n \ge 6$ is not a power of an odd prime nor equal to $2^k + 2$ for some $k$, then $\exp(A_n) = \exp(A_{n-1})$, and so $A_{n-1} \times A_n$ is \expsimp{} for all such $n$ by Theorem \ref{thm:charIIPsimple}.
\end{ex}

Corollaries \ref{cor:simpleIIPsimple}, \ref{cor:solvIIPsimple}, and \ref{cor:directprod_IIPsimple} might lead one to believe that all composition factors of an \expsimp{} group need to have the same exponent.  This is not the case.  Recall that a group $G$ is \textit{quasisimple} if $G = [G,G]$ and $G/Z(G)$ is a simple group.

\begin{cor}
 \label{cor:quasiIIP}
Let $G$ be quasisimple group. Then, $G$ is \expsimp{} if and only if $\exp(G) = \exp(G/Z(G))$.
\end{cor}

\begin{proof}
If $G$ is \expsimp{}, then $\exp(G) = \exp(G/Z(G))$ by Theorem \ref{thm:charIIPsimple}. Conversely, assume $\exp(G) = \exp(G/Z(G))$.  Let $N$ be a proper normal subgroup of $G$.  Since $G$ is quasisimple, $N \le Z(G)$.  Moreover, since $1 \le N \le Z(G)$, 
 \[ \exp(G/Z(G)) \le \exp(G/N) \le \exp(G).\]
Thus, $\exp(G/N) = \exp(G)$ and $G$ is \expsimp{}.
\end{proof}

\begin{ex}
 Consider the quasisimple group $G = 3.A_6$ (this is SmallGroup(1080,260) in GAP \cite{GAP4}).  Note that $Z(G) \cong C_3$ and $G/Z(G) \cong A_6$, so $G$ has distinct composition factors $C_3$ and $A_6$.  Since $\exp(G) = \exp(A_6) = 60$, $G$ is \expsimp{} by Corollary \ref{cor:quasiIIP}.  Thus, an \expsimp{} group can have composition factors with distinct exponents.  
\end{ex}

%%%%%%%%%%%%%%%%%%%%%%%%%%%%%%%%%%%%%%%%%%%%%%%%%
%%%%%%%%%%%%%%%%%%%%%%%%%%%%%%%%%%%%%%%%%%%%%%%%%
%%%%%%%%%%%%%%%%%%%%%%%%%%%%%%%%%%%%%%%%%%%%%%%%%

\section{Preliminary results related to \wexp{} subgroups}
\label{sect:prelims}

Recall that a subgroup $H$ of $G$ is \textit{\wexp{}} (denoted $H \lwexp G$) if, for all $x \in G$, there exists $g \in G$ such taht $x^{|G:H|} \in H^g$. We say that $G$ is \textit{\wic{}} if $H \lwexp G$ for all subgroups $H$ of $G$.  Here, we collect some lemmas that are useful for determining whether or not a group is \wic.

\begin{lem}
 \label{lem:goodmax}
 Let $G$ be a finite group. Let $Y$ be a subgroup of $G$ such that $Y$ is \wic{} and $Y$ is \wexp{} with respect to $G$.  If $H \le Y$, then $H$ is \wexp{} with respect to $G$.
\end{lem}

\begin{proof}
 Let $G$, $Y$, $H$ be as in the statement of the lemma.  Since $Y \lwexp G$, for each $x \in G$ there exists $g \in G$ such that $x^{|G:Y|} \in Y^g$.  Since $Y^g \cong Y$, $Y^g$ is \wic, i.e., there exists $y \in Y^g$ such that 
 \[ x^{|G:H|} = \left( x^{|G:Y|}\right)^{|Y^g:H^g|} \in H^{gy},\]
 which shows that $H \lwexp G$.
\end{proof}

\begin{lem}
\label{lem:checkmax}
 Let $G$ be a group such that all maximal subgroups of $G$ are themselves \wic{} groups.  Then, $G$ is \wic{} if and only if each maximal subgroup is \wexp{} with respect to $G$. 
\end{lem}

\begin{proof}
 First, if $M$ is maximal and $M$ is not \wexp{} with respect to $G$, then $G$ is not \wic{} by definition.
 
 Conversely, assume all maximal subgroups of $G$ are \wexp{} with respect to $G$ and are themselves \wic{} groups.  Let $x \in G$ and $H < G$.  Since $H < G$, there exists a maximal subgroup $M$ of $G$ such that $H \le M < G$.  Since $M \lwexp G$ and $M$ is \wic, $H \lwexp G$ by Lemma \ref{lem:goodmax}.  Since $H$ was arbitrary, $G$ is \wic.
\end{proof}

\begin{lem}
 \label{lem:wicquo}
 If $G$ is a \wic{} group, then all quotients of $G$ are \wic.
\end{lem}

\begin{proof}
 Let $G$ be a \wic{} group, and let $\phi: G \to \overline{G}$ be a surjective homomorphism.  Let $\overline{x} \in \overline{G}$ and $\overline{H} \le \overline{G}$.  Define $H \colonequals \phi^{-1}(\overline{H}) \le G$, and let $x$ be any element of $G$ in $\phi^{-1}(\overline{x})$.  Since $G$ is \wic, there exists $g \in G$ such that 
 \[ x^{|G:H|} \in H^g,\]
 and hence 
 \[ \overline{x}^{|\overline{G}:\overline{H}|} = \overline{x}^{|G:H|} \in \overline{H}^{\overline{g}}.\]
 Thus, $\overline{H} \lwexp \overline{G}$.  Since $\overline{H}$ was arbitrary, $\overline{G}$ is \wic.
\end{proof}

%%%%%%%%%%%%%%%%%%%%%%%%%%%%%%%%%%%%%%%%%%%%%%%%%
%%%%%%%%%%%%%%%%%%%%%%%%%%%%%%%%%%%%%%%%%%%%%%%%%
%%%%%%%%%%%%%%%%%%%%%%%%%%%%%%%%%%%%%%%%%%%%%%%%%

\section{Solvable groups are \wic}
\label{sect:solv}

The goal of this section is to prove that the issues raised in Example \ref{ex:easy} are essentially all that prevent subgroups of solvable groups from being exponential; that is, all subgroups of solvable groups are \wexp{} (and hence all solvable groups are \wic{}, justifying the terminology).  To prove this, we need the following lemma on the exponent of a Sylow $p$-subgroup of the affine general linear group $\AGL(n,p)$.

%While it is well known that the exponent of a Sylow $p$-subgroup of the affine general linear group $\AGL(n,p)$ is $p^{\lceil \log_p n \rceil + 1}$, what we actually need is the bound listed in the following result.
%To prove this, we first need the following lemma.  While it is well known that the exponent of a Sylow $p$-subgroup of the affine general linear group $\AGL(n,p)$ is $p^{\lceil \log_p n \rceil + 1}$, what we actually need is the bound listed in the following result.

\begin{lem}
 \label{lem:expAGL}
 Let $P$ be a Sylow $p$-subgroup of $\AGL(n,p)$.  Then, $\exp(P) \le p^n$.
\end{lem}

\begin{proof}
 Since $\AGL(n,p) = C_p^n \rtimes \GL(n,p)$, $P \cong C_p^n \rtimes Q$, where $Q$ is a Sylow $p$-subgroup of $\GL(n,p)$.  Thus, $\exp(P) \le \exp(Q) \cdot p$.
 
 Let $g \in \GL(n,p)$.  By the Cayley-Hamilton Theorem, $g$ satisfies $\chi(x)$, where $\chi(x) \colonequals \det(xI - g)$ is the characteristic polynomial of $g$.  Thus, $g$ generates a subalgebra of $\GF(p)[g]$ containing at most $p^n - 1$ nonzero elements, which implies $|g| \le p^n - 1$.  If $g \in Q$, this means $|g| \le p^{n-1}$.  The result follows.
\end{proof}

We are now ready to prove that all solvable groups are \wic.

\begin{proof}[Proof of Theorem \ref{thm:solvwic}]
 Assume $G$ is a finite solvable group.  Note that, since all subgroups of nilpotent groups are exponential, all nilpotent groups are \wic.  We proceed by induction on $|G|$; that is, assume all solvable groups with order less than $|G|$ are \wic.  By Lemma \ref{lem:checkmax}, it suffices to prove that all maximal subgroups of $G$ are \wexp.  Let $M$ be a maximal subgroup of $G$.
 
 Let $N$ be a nontrivial normal subgroup of $G$, and let $x \in G$. Assume first that $N \le M$. Let $\overline{G} = G/N$ and $\overline{M} = M/N \le \overline{G}$. By the inductive hypothesis, $\overline{G}$ is \wic, so 
\begin{equation}\label{eq1}
\overline{x}^{|\overline{G}:\overline{M}|} \in \overline{M}^{\overline{g}}
\end{equation}
 for some $Ng = \overline{g} \in \overline{G}$, where $\overline{x} = Nx$ and $g \in G$. Since $|\overline{G}:\overline{M}| = |G:M|$, \eqref{eq1} implies that $x^{|G:M|} \in M^g$, 
 showing that $M$ is \wexp{} with respect to $G$.  

Now, assume that $N \not\le M$ for all nontrivial normal subgroups $N$ of $G$.  Since $G$ is solvable, it contains a minimal normal subgroup $N \cong C_p^n$ (see, for example, \cite[Lemma 3.11]{Isaacs_2008}), where $n$ is a positive integer and $p$ is a prime.  Since $M$ is maximal and $N \not\le M$, we have $G = NM$.  Moreover, $M \cap N$ is a proper subgroup of $N$ that is normal in both $M$ and $N$ (since $N$ is elementary abelian), and hence $M \cap N$ is normal in $G$; that is, $M \cap N = \{1\}$ and $G = N \rtimes M$.  Furthermore, since $N \cong C_p^n$, there is a natural homomorphism $\phi: M \to \GL(n,p)$.  Since $\Ker(\phi)$ is a normal subgroup of $M$ that centralizes $N$, $\Ker(\phi)$ is a normal subgroup of $G$ that is contained in $M$. By assumption, $\Ker(\phi) = \{1\}$, and so $\phi$ is faithful. Thus, $M \liso \GL(n,p)$ and $G \liso \AGL(n,p)$.
 
 Let $x \in G$, and suppose $|x| = p^j \cdot k$, where $\gcd(p,k) = 1$.  There exist $s,t \in \Z$ such that $p^j s + kt = 1$, so
 \[ x = ( x^{p^j})^s (x^k )^t \in \langle x^{p^j}, x^k \rangle. \]
Thus, to conclude that $M$ is \wexp{} with respect to $G$, it suffices to show that there is a conjugate of $M$ containing both $(x^{p^j})^{|G:M|}$ and $(x^k)^{|G:M|}$.

First, consider $x^k$. Since $|x^k| = p^j$, $x^k$ is contained in a Sylow $p$-subgroup of $G \liso \AGL(n,p)$.  By Lemma \ref{lem:expAGL}, 
\begin{equation}\label{x^k eq}
(x^k)^{|G:M|} = (x^k)^{p^n} = 1.
\end{equation}
Next, we work with $x^{p^j}$, which has order $k$.  Since $k$ is coprime to $p$ and $G$ is solvable, $x^{p^j}$ is contained in a Hall $p'$-subgroup $H$ of $G$ \cite[Theorem 3.13]{Isaacs_2008}.  Furthermore, since $|M| = |G|/p^n$, $M$ contains a Hall $p'$-subgroup $L$ of $G$.  By Hall's Theorem \cite[Theorem 3.14]{Isaacs_2008}, $H$ and $L$ are conjugate in $G$. So, there exists $g \in G$ such that $L^g = H$, and hence
\begin{equation}\label{x^{p^j} eq}
x^{p^j} \in H = L^g \subseteq M^g.
\end{equation}
By \eqref{x^k eq} and \eqref{x^{p^j} eq}, both $(x^{p^j})^{|G:M|}$ and $(x^k)^{|G:M|}$ are in $M^g$, which completes the proof.
\end{proof}

%%%%%%%%%%%%%%%%%%%%%%%%%%%%%%%%%%%%%%%%%%%%%%%%%
%%%%%%%%%%%%%%%%%%%%%%%%%%%%%%%%%%%%%%%%%%%%%%%%%
%%%%%%%%%%%%%%%%%%%%%%%%%%%%%%%%%%%%%%%%%%%%%%%%%

\section{Almost all symmetric and alternating groups are \indexc}
\label{sect:SnAn}

The purpose of this section is to determine which symmetric and alternating groups are \indexc. Apart from three exceptions ($A_5$, $A_6$, and $S_6$), all symmetric and alternating groups that are not solvable are \indexc. 

In the course of proving that a group $G$ is \wic, we will frequently use the following elementary lemma to help verify that a subgroup of $G$ is \wexp{} with respect to $G$.

\begin{lem}\label{lem:sledgehammer}
Let $H \le G$ and let $y \in G$ be such that $|y|=p^k$ for some prime $p$ and some $k \ge 0$. If $H$ contains a Sylow $p$-subgroup of $G$, then $y \in H^g$ for some $g \in G$.
\end{lem}
\begin{proof}
The element $y$ is contained in a Sylow $p$-subgroup of $G$, and all such subgroups are conjugate in $G$.
\end{proof}

We begin by considering the cases when $n = 5$ or $6$.    

\begin{lem}
 \label{lem:A5}
 The group $A_5$ is \wic.
\end{lem}

\begin{proof}
 The maximal subgroups of $G = A_5$ are isomorphic to $S_3$ (index $10$, one conjugacy class), $D_{5}$ (index $6$, one conjugacy class; here, $D_n$ indicates a dihedral group of order $2n$), and $A_4$ (index $5$, one conjugacy class).  These groups are all solvable, and hence \wic{} by Theorem \ref{thm:solvwic}. So, by Lemma \ref{lem:checkmax}, it suffices to verify that each of these subgroups is \wexp{} with respect to $G$. Note that elements of $G$ have order 1, 2, 3 or 5.

Let $M$ be a maximal subgroup of $G$ and let $x \in G$. If $M \cong S_3$, then $|x^{|G:M|}| \in \{1, 3\}$ and $M$ contains a Sylow 3-subgroup of $G$. If $M \cong D_5$, then $|x^{|G:M|}| \in \{1, 5\}$ and $M$ contains a Sylow 5-subgroup of $G$. Finally, if $M \cong A_4$, then $|x^{|G:M|}| \in \{1, 2, 3\}$ and $M$ contains both a Sylow 2-subgroup and a Sylow 3-subgroup of $G$. In all cases, by Lemma \ref{lem:sledgehammer} we have $x^{|G:M|} \in M^g$ for some $g \in G$. Thus, every maximal subgroup of $G$ is \wexp{} with respect to $G$, as desired.
 \end{proof}
 
 \begin{lem}
 \label{lem:A6}
 The group $A_6$ is \wic.
\end{lem}

\begin{proof}
Let $G = A_6$.  The maximal subgroups of $G$ are isomorphic to $S_4$ (index $15$, two conjugacy classes), $(C_3 \times C_3) \rtimes C_4$ (index $10$, one conjugacy class), and $A_5$ (index $6$, two conjugacy classes). Every element $x$ of $G$ is contained in a cyclic group of order $3$, $4$, or $5$. Let $M$ be a maximal subgroup of $G$. If $M \cong S_4$, then $|x^{|G:M|}| \in \{1, 2, 4\}$ and $M$ contains a Sylow 2-subgroup of $G$, so Lemma \ref{lem:sledgehammer} applies in this case. If $M \cong (C_3 \times C_3) \rtimes C_4$, then $|x^{|G:M|}| \in \{1, 2, 3\}$. This time, $M$ contains a Sylow 3-subgroup of $G$, but not a Sylow 2-subgroup of $G$. However, all elements of order 2 are conjugate in $G$, so $x^{|G:M|} \in M^g$ for some $g \in G$ regardless of the order of $x^{|G:M|}$. Lastly, if $M \cong A_5$, then $|x^{|G:M|}| \in \{1, 2, 5\}$ and $M$ contains a Sylow 5-subgroup of $G$. Here, we may proceed as in the prior case to conclude that $M \lwexp G$.
 \end{proof}
 
 \begin{lem}
  \label{lem:S5}
  The group $S_5$ is \IEe.
 \end{lem}

 \begin{proof}
 Let $G = S_5$.  This group is not \wic. Take $x = (1 \; 2 \; 3)(4 \; 5)$ and $M = \Sym(\{1,2,3,4\}) = S_4$.  Then, $|x| = 6$ and $|G:M| = 5$, so $|x^{|G:M|}| = 6$.  Since $S_4$ does not contain an element of order $6$, $M$ is not \wexp{} with respect to $G$.  The only proper quotient of $G$ is solvable and thus \wic{} by Theorem \ref{thm:solvwic}, so $G$ is \IEe.
 \end{proof}
 
 \begin{lem}
 \label{lem:S6}
 The group $S_6$ is \wic{}.
\end{lem}

\begin{proof}
 Let $G = S_6$.  The maximal subgroups of $G$ are isomorphic to $S_4 \times S_2$ (index $15$, two conjugacy classes), $S_3 \Wr S_2$ (index $10$, one class), $S_5$ (index $6$, two classes), and $A_6$ (index $2$, one class). Furthermore, $G$ contains three classes of elements of order $2$, two classes of order $3$, two classes of order $4$, one class of order $5$, and two classes of order $6$. We will begin by checking that $M \lwexp G$ for each maximal subgroup $M$ of $G$.  Let $x \in G$.

If $M \cong A_6$, then $M$ is normal in $G$ and hence $M \lexp G$. If $M \cong S_4 \times S_2$, then $|x^{|G:M|}| \in \{1,2,4\}$ and $M$ contains a Sylow 2-subgroup of $G$, so we may use Lemma \ref{lem:sledgehammer}. 

Next, if $M \cong S_3 \Wr S_2$, then $|x^{|G:M|}| \in \{1,2,3\}$ and $M$ contains a Sylow 3-subgroup of $G$, so Lemma \ref{lem:sledgehammer} applies if $|x^{|G:M|}| \ne 2$. On the other hand, if $|x^{|G:M|}|=2$, then $x$ is either a $4$-cycle or the product of a $4$-cycle and a $2$-cycle; either way, $x^{|G:M|}$ is a product of two disjoint $2$-cycles. Such a product stabilizes a decomposition of $\{1, \dots, 6\}$ into two sets of size $3$, so a product of two disjoint $2$-cycles is contained in a conjugate of $M$ in $G$.  Thus, $M \lwexp G$.

Finally, consider the case where $M \cong S_5$. This time, $|x^{|G:M|}| \in \{1,2,5\}$ and $M$ contains a Sylow 5-subgroup of $G$. If $|x^{|G:M|}| = 2$, then, by similar reasoning as above, $x^{|G:M|}$ is a product of two disjoint $2$-cycles.  Thus, $x^{|G:M|}$ fixes an element of $\{1,\dots, 6\}$ and hence is contained in the stabilizer of that element in $G$, which is isomorphic to $S_5$ and a conjugate of $M$.

At this point, we know that every maximal subgroup of $G$ is \wexp{} with respect to $G$. Since the maximal subgroups isomorphic to one of $S_4 \times S_2$, $S_3 \Wr S_2$, or $A_6$ are \wic{} (Theorem \ref{thm:solvwic} or Lemma \ref{lem:A6}), by Lemma \ref{lem:goodmax}, any subgroup of $G$ that is contained in a maximal subgroup isomorphic to one of these is \wexp{} with respect to $G$. The only proper subgroups of $G$ that are not maximal and not contained in a subgroup isomorphic to one of $S_4 \times S_2$, $S_3 \Wr S_2$, or $A_6$ are isomorphic to $C_5 \rtimes C_4$ and have index $36$ in $G$. Let $H \cong C_5 \rtimes C_4$ be such a subgroup. Then, $|x^{|G:H|}| \in \{1,5\}$ and $H$ contains a Sylow 5-subgroup of $G$. Thus, $H \lwexp G$ by Lemma \ref{lem:sledgehammer}, which completes the proof.
\end{proof}

We are now ready to prove the classification of \IEe{} groups that are symmetric or alternating, starting with symmetric groups.

\begin{proof}[Proof of Theorem \ref{thm:Sn}]
 Let $G = S_n$.  When $n \le 4$, $S_n$ is solvable and hence \wic{} by Theorem \ref{thm:solvwic}.  By Lemmas \ref{lem:S5} and \ref{lem:S6}, the group $S_5$ is \IEe{} and the group $S_6$ is \wic.  So, assume $n \ge 7$.

We may express $n$ as $n=k+\ell$, where $\gcd(k,n) = \gcd(\ell, n) = 1$ and $\ell \ge k \ge 3$. Explicitly, when $n$ is odd, we can take $k = \tfrac{n-1}{2}$ and $\ell = \tfrac{n+1}{2}$; when $n$ is divisible by $4$, we may set $k = \tfrac{n}{2}-1$ and $\ell = \tfrac{n}{2}+1$; and when $n \equiv 2 \pmod 4$, we can let $k = \tfrac{n}{2}-2$ and $\ell = \tfrac{n}{2}+2$

Take $x = (1 \; 2 \; \dots \; k)(k+1 \; \dots \; n)$, the product of a disjoint $k$-cycle with a disjoint $\ell$-cycle, and let $H = S_{n-1}$.  Since $|G:H| = n$, $x^{|G:H|}$ is still the product of a disjoint $k$-cycle with a disjoint $\ell$-cycle, and hence $x^{|G:H|}$ does not fix any points in $\{1, \dots , n\}$.  On the other hand, the conjugates of $H$ in $G$ are the stabilizers of points in the natural action, so $x^{|G:H|}$ cannot be contained in any conjugate of $H$.  Thus, $H$ is not \wexp{} with respect to $G$, and, when $n \ge 7$, $S_n$ is \indexc.  The only proper quotient of $S_n$ is $C_2$ when $n \ge 7$, which is solvable and \wic, and so $S_n$ is indeed \IEe{} exactly when $n = 5$ or $n \ge 7$. 
\end{proof}

To end this section, we prove the classification of the \IEe{} alternating groups.

\begin{proof}[Proof of Theorem \ref{thm:An}]
 The groups $A_5$ and $A_6$ are \wic{} by Lemmas \ref{lem:A5} and \ref{lem:A6}, respectively.  Assume $n \ge 7$, and let $G = A_n$.
 
 If $n$ is even, then the choice of $x \in S_n$ in the proof of Theorem \ref{thm:Sn} is the product of two odd cycles, and hence $x \in A_n$.  If we choose $H = A_{n-1}$, then $|G:H| = n$.  The proof that $x^{|G:H|} \notin H^g$ for any $g \in G$ is analogous to the proof for $S_n$, and $H$ is not \wexp{} with respect to $G$.  Thus, $G$ is \indexc{} when $n$ is even.
 
 If $n$ is odd, then $n -4$ is odd and $\gcd(n-4, n) = 1$.  Consider
 \[ x = (1 \; 2 \; \dots \; n -4)(n-3 \; n - 2)(n-1 \; n).\]
 If we choose $H = A_{n-1}$, then $|G:H| = n$, and the conjugates of $H$ in $G$ are the stabilizers of a point in the natural action.  On the other hand, $x^{|G:H|}$ is still the product of an $(n-4)$-cycle with two disjoint $2$-cycles, so $x^{|G:H|}$ fixes no point in the natural action of $A_n$ on $\{1,\dots, n\}$ and is not in any conjugate of $H$ in $G$.  Thus, $H$ is not \wexp{} with respect to $G$, and $G$ is \indexc. Since $G$ is simple, it is in fact \IEe.
\end{proof}

%%%%%%%%%%%%%%%%%%%%%%%%%%%%%%%%%%%%%%%%%%%%%%%%%
%%%%%%%%%%%%%%%%%%%%%%%%%%%%%%%%%%%%%%%%%%%%%%%%%
%%%%%%%%%%%%%%%%%%%%%%%%%%%%%%%%%%%%%%%%%%%%%%%%%

\section{\Wic{} groups isomorphic to $\PSL(2,q)$}
\label{sect:PSL}

\subsection{Structure of $\PSL(2,q)$}

The maximal subgroups of $\PSL(2,q)$ were classified by Dickson \cite{Dickson_1958}, where $q$ is a prime power.  Here, we split the result into cases depending on $q$ (and exclude some small groups, which exhibit exceptional behavior). In what follows, the dihedral group $D_n$ has order $2n$, and, unless otherwise stated, all maximal subgroups in a particular class are conjugate.

\begin{thm}
\label{thm:maxPSL}
 Let $q = p^d \ge 4$ be a prime power, $q \neq 5, 7, 9, 11$, and let $e = \gcd(p-1, 2)$.  Then, the maximal subgroups of $G = \PSL(2,q)$ are listed in Tables \ref{tbl:PSLmax1} (which applies to all $q$), \ref{tbl:PSLmax2} ($q = p$ prime), \ref{tbl:PSLmax3} ($q$ even), \ref{tbl:PSLmax4} ($q > p$, $q$ odd). 
\end{thm}

\begin{center}
 \bgroup
\def\arraystretch{1.5}
\begin{table}[H]
 \begin{tabular}{|c|c|c|c|}
 \hline
  Structure of $M$ & Conditions on $q$ & $|G:M|$ & $\#$ Conj. Classes\\
  \hline
  \hline
  $C_p^d \rtimes C_{(q-1)/e}$& all $q$ & $q+1$ & $1$ \\
  $D_{(q-1)/e}$ & all $q$ & $\frac{q(q+1)}{e}$ & $1$ \\
  $D_{(q+1)/e}$& all $q$ & $\frac{q(q-1)}{e}$ & 1 \\
  \hline
 \end{tabular}
 \caption{Maximal subgroups of $G = \PSL(2,q)$, all $q \ge 4$, $q \neq 5,7,9,11$.}
 \label{tbl:PSLmax1}
\end{table}
\egroup
\end{center}

\begin{center}
\bgroup
\def\arraystretch{1.5}
\begin{table}[H]
 \begin{tabular}{|c|c|c|c|}
 \hline
  Structure of $M$ & Conditions on $p$ & $|G:M|$ & $\#$ Conj. Classes\\
  \hline
  \hline
  $A_5$ & $p \equiv \pm 1 \pmod {10}$ & $\frac{(p-1)p(p+1)}{120}$ & 2 \\
  $A_4$ & $p \equiv \pm 3 \pmod 8$, $p \not\equiv \pm 1 \pmod {10}$ & $\frac{(p-1)p(p+1)}{24}$ & 1 \\
  $S_4$ & $p \equiv \pm 1 \pmod 8$ & $\frac{(p-1)p(p+1)}{48}$ & 2\\
  \hline
 \end{tabular}
 \caption{Other maximal subgroups of $G = \PSL(2,p)$, $p$ prime, $p \ge 13$.}
 \label{tbl:PSLmax2}
\end{table}
\egroup
\end{center}

\begin{center}
\bgroup
\def\arraystretch{1.5}
\begin{table}[H]
 \begin{tabular}{|c|c|c|c|}
 \hline
  Structure of $M$ & Conditions on $q$ & $|G:M|$ & $\#$ Conj. Classes\\
  \hline
  \hline
  $\PGL(2,q_0) \cong \PSL(2,q_0)$ & $q = q_0^r$, $r$ prime, $q_0 > 2$ & $\frac{(q-1)q(q+1)}{(q_0 - 1)q_0(q_0 + 1)}$ & 1 \\
  \hline
 \end{tabular}
 \caption{Other maximal subgroups of $G = \PSL(2,q)$, $q > 2$ even.}
 \label{tbl:PSLmax3}
\end{table}
\egroup
\end{center}

\begin{center}
\bgroup
\def\arraystretch{1.5}
\begin{table}[H]
 \begin{tabular}{|c|c|c|c|}
 \hline
  Structure of $M$ & Conditions on $q$ & $|G:M|$ & $\#$ Conj. Classes\\
  \hline
  \hline
  $A_5$ & $p \equiv \pm 3 \pmod {10}$ and $q = p^2$ & $\frac{(q-1)q(q+1)}{120}$ & 2 \\
  $\PGL(2,q_0)$ & $q = q_0^2$ & $\frac{(q-1)q(q+1)}{2(q_0 - 1)q_0(q_0 + 1)}$ & 2 \\
  $\PSL(2,q_0)$ & $q = q_0^r$, $r$ prime and odd & $\frac{(q-1)q(q+1)}{(q_0 - 1)q_0(q_0 + 1)}$ & 1 \\
  \hline
 \end{tabular}
 \caption{Other maximal subgroups of $G = \PSL(2,q)$, $q = p^d \ge 25$, $q$ odd.}
 \label{tbl:PSLmax4}
\end{table}
\egroup
\end{center}

\begin{lem}
 \label{lem:maxPSLwic}
 Let $p \ge 13$ be a prime.  All maximal subgroups of $\PSL(2,p)$ are \wic{} groups.
\end{lem}

\begin{proof}
 By Theorem \ref{thm:maxPSL}, all maximal subgroups of $\PSL(2,p)$, $p \ge 13$, are either solvable or isomorphic to $A_5$.  In either case, these groups are \wic.
\end{proof}

% \begin{lem}\cite[Lemma 3.1]{Bryce_Fedri_Serena_1999}
%  \label{lem:PSLcover}
%  Let $p > 11$ be an odd prime.  Every element of $\PSL(2,p)$ is contained in either the stabilizer of a projective point or the normalizer of a Singer cycle.
% \end{lem}

We also highlight some well-known facts about conjugacy classes of subgroups and elements of $\PSL(2,q)$ \cite{Dickson_1958}.

\begin{lem}\label{lem:PSLcyclic}
Let $p$ be prime, $d \ge 1$, and $q = p^d \ge 4$.
\begin{enumerate}[(1)]
 \item If $e = \gcd(p-1,2)$, then every element of $\PSL(2,q)$ is contained in a cyclic subgroup isomorphic to $C_{(q-1)/e}$, $C_p$, or $C_{(q+1)/e}$.  Moreover, all cyclic groups of a given order are conjugate in $\PSL(2,q)$.
 \item If $p$ is odd, then $\PSL(2,p)$ contains a unique conjugacy class of elements of order $2$, a unique conjugacy class of elements of order $3$, and at most one conjugacy class of elements of order $4$.
\end{enumerate}
\end{lem}

Because of Lemma \ref{lem:PSLcyclic}(1), when checking whether a subgroup of $\PSL(2,q)$ is \wexp, it is enough to focus on elements of order $p$, $(q-1)/e$, or $(q+1)/e$.

The following result applies for all prime powers $q$ such that $q \ge 4$, $q \neq 5, 7, 9, 11$. At in Theorem \ref{thm:maxPSL}, we omit these cases due to the exceptional behavior of maximal subgroups.

\begin{prop}
 \label{prop:PSLmaxgood}
 Let $p$ be prime, $d \ge 1$, $q = p^d \ge 4$, $q \neq 5, 7, 9, 11$, and let $e = \gcd(p-1, 2)$.  If $G = \PSL(2,q)$ and $M \le G$ is isomorphic to one of $C_p^d \rtimes C_{(q-1)/e}$, $D_{(q-1)/e}$, or $D_{(q+1)/e}$, then $M \lwexp G$.
\end{prop}

\begin{proof}
 By Lemma \ref{lem:PSLcyclic}, every element of $G$ is contained in a cyclic subgroup isomorphic to $C_{(q-1)/e}$, $C_p$, or $C_{(q+1)/e}$, so we need only check elements $x$ with orders $p$, $(q-1)/e$, or $(q+1)/e$ in $G$ for each $M$.  
 
 Suppose first that $|x| = p$.  If $M \cong C_p^d \rtimes C_{(p-1)/e}$, then $x$ is in a conjugate of $M$, and we are done.  Otherwise, $|G:M|$ is divisible by $p$, and so the result holds in this case.
 
 Next, suppose $|x| = (q-1)/e$.  If $M \cong C_p^d \rtimes C_{(p-1)/e}$ or $M \cong D_{(q-1)/e}$, then $x$ is in a conjugate of $M$, and we are done.  Otherwise, $|G:M|$ is divisible by $(q-1)/e$, and so the result holds in this case.
 
 Finally, suppose $|x| = (q+1)/e$.  If $M \cong D_{(q+1)/e}$, then $x$ is in a conjugate of $M$, and we are done.  Otherwise, $|G:M|$ is divisible by $(q+1)/e$, proving the result.
\end{proof}

\subsection{The primes for which $\PSL(2,p)$ is \wic}

The purpose of this section is to prove Theorem \ref{thm:PSLwic} in the case when $q = p$.  We will need a few preliminary results to do so.

\begin{prop}
 \label{prop:PSLwicA4S4A5}
 Let $p \ge 13$ be prime.  The group $G = \PSL(2,p)$ is \wic{} if and only if the maximal subgroups isomorphic to one of $A_4$, $S_4$, or $A_5$ are \wexp{} with respect to $G$. 
\end{prop}

\begin{proof}
 By Lemma \ref{lem:maxPSLwic}, every maximal subgroup of $G$ is \wic, so, by Lemma \ref{lem:checkmax}, $G$ is \wic{} if and only if all maximal subgroups are \wexp{} with respect to $G$.  That we need only check subgroups isomorphic to $A_4$, $A_5$, or $S_4$ follows from Theorem \ref{thm:maxPSL} and Proposition \ref{prop:PSLmaxgood}. 
\end{proof}

\begin{lem}\label{lem:chainsaw}
Let $p \ge 13$ be prime, let $G = \PSL(2,p)$, and let $M$ be a maximal subgroup of $G$ isomorphic to one of $A_5$, $A_4$, or $S_4$.
\begin{enumerate}[(1)]
\item If $a \in G$ and $|a|=p$, then $a^{|G:M|} = 1$.
\item There exist unique even integers $k$ and $\ell$ such that
\begin{enumerate}[(i)]
\item $2|M|=k\ell$, and
\item $(p-1)/k$ and $(p+1)/\ell$ are coprime integers.
\end{enumerate}
\item With $k$ and $\ell$ as in part (2), let $x, y \in G$ such that $|x|=(p-1)/2$ and $|y|=(p+1)/2$. Then, $|x^{|G:M|}| = k/2$ and $|y^{|G:M|}| = \ell/2$.
\end{enumerate}
\end{lem}
\begin{proof}
Note that $|G:M| = (p-1)p(p+1)/(2|M|)$ and the only primes that could divide $|M|$ are 2, 3, and 5.
\begin{enumerate}[(1)]
\item Since $p \ge 13$, $p$ is coprime to $2|M|$, $p-1$, and $p+1$. Thus, $\gcd(p, |G:M|)=p$, and so $a^{|G:M|}=1$ whenever $|a|=p$.
\item Certainly, $(p-1)(p+1)/(2|M|)$ is an integer, and the integers $p-1$ and $p+1$ have no odd prime factors in common. Moreover, exactly one of $p-1$ or $p+1$ is congruent to $2 \pmod 4$, and the other is congruent to $0 \pmod 4$. Without loss of generality, assume that $p-1 \equiv 2 \pmod 4$. Factor $2|M|$ as $2|M| = 2^e m$, where $e \geq 3$ and $m$ is odd. Take $k=2\gcd(p-1,m)$ and $\ell = 2^{e-1}\gcd(p+1,m)$. Then, $k$ and $\ell$ have all of the required properties. Note that these integers are unique, since $k$ must be congruent to $2 \pmod 4$, $\ell$ must be congruent to $0 \pmod 4$, and $p-1$ and $p+1$ share no odd prime factors.
\item We may factor $|G:M|$ as
\begin{equation*}
|G:M| = \dfrac{p-1}{k} \cdot p \cdot \dfrac{p+1}{\ell},
\end{equation*}
where $(p-1)/k$ and $(p+1)/\ell$ are coprime. Then,
\begin{equation*}
\gcd\Big(\dfrac{p-1}{2}, \; |G:M|\Big) = \dfrac{p-1}{k} \quad \text{and} \quad \gcd\Big(\dfrac{p+1}{2}, \; |G:M|\Big) = \dfrac{p+1}{\ell},
\end{equation*}
which implies that $|x^{|G:M|}| = k/2$ and $|y^{|G:M|}| = \ell/2$. \qedhere
\end{enumerate}
\end{proof}

\begin{lem}
 \label{lem:A5inPSLbad}
 Let $p \ge 13$ be prime.  If $p \equiv \pm 1 \pmod {10}$, $G = \PSL(2,p)$, and $M \cong A_5$ is a subgroup of $G$, then there exists an element $x \in G$ of order $(p-1)/2$ or $(p+1)/2$ such that the order of $x^{|G:M|}$ is at least $6$.  In particular, $M$ is not \wexp{} with respect to  $G$, and $G$ is \indexc.
\end{lem}
\begin{proof}
Assume $p \equiv \pm 1 \pmod{10}$. By Lemma \ref{lem:PSLcyclic}, it suffices to consider elements of $G$ of order $p$, $(p-1)/2$, and $(p+1)/2$. From Lemma \ref{lem:chainsaw}, we know that $a^{|G:M|} \in M$ whenever $|a|=p$. For elements of other orders, we will consider cases based on the value of $p \pmod{12}$. The tables below summarize the possibilities; the notation is as in Lemma \ref{lem:chainsaw}.
\begin{center}
\begin{minipage}{0.5\textwidth}
\begin{table}[H]
\begin{tabular}{ccccc}
$p \pmod{12}$ & $k$ & $\ell$ & $|x^{|G:M|}|$ & $|y^{|G:M|}|$\\
\hline
1  & 60 &  2 & 30 & 1\\
5  & 20 &  6 & 10 & 3\\
7  & 30 &  4 & 15 & 2\\
11 & 10 & 12 &  5 & 6
\end{tabular}
\caption{$M \cong A_5$, $p \equiv 1 \pmod{10}$}
\label{tbl:A5_1}
\end{table}
\end{minipage}%
\begin{minipage}{0.5\textwidth}
\begin{table}[H]
\begin{tabular}{ccccc}
$p \pmod{12}$ & $k$ & $\ell$ & $|x^{|G:M|}|$ & $|y^{|G:M|}|$\\
\hline
1  & 12 & 10 & 6 &  5\\
5  &  4 & 30 & 2 & 15\\
7  &  6 & 20 & 3 & 10\\
11 &  2 & 60 & 1 & 30
\end{tabular}
\caption{$M \cong A_5$, $p \equiv 9 \pmod{10}$}
\label{tbl:A5_9}
\end{table}
\end{minipage}
\end{center}
Since $A_5$ contains no elements of order 6, 15, 10, or 30, we see that in each case either $x^{|G:M|}$ or $y^{|G:M|}$ fails to be in any conjugate of $M$. Therefore, $M$ is not \wexp{} with respect to $G$, and $G$ is \indexc.
\end{proof}

\begin{lem}
 \label{lem:whenA4good}
 Let $p \ge 13$ be prime.  Let $p \equiv \pm 3 \pmod {10}$ and $p \equiv \pm 3 \pmod 8$.  Then, a maximal subgroup $M \cong A_4$ is \wexp{} with respect to $G = \PSL(2,p)$ if and only if $p \equiv \pm 5 \pmod {24}$. In particular, if $p \not\equiv \pm 5 \pmod {24}$, then there exists an element $x \in G$ of order $(p-1)/2$ or $(p+1)/2$ such that the order of $x^{|G:M|}$ is $6$, and a maximal subgroup $M \cong A_4$ is \wexp{} with respect to $G$ if and only if
 \[ p \equiv 43, 53, 67, 77 \pmod {120}.\]
\end{lem}
\begin{proof}
As in Lemma \ref{lem:A5inPSLbad}, it suffices to consider elements of $G$ of order $(p-1)/2$ or $(p+1)/2$. We will consider cases depending on the value of $p \pmod{3}$, and use the notation of Lemma \ref{lem:chainsaw}. Tables \ref{tbl:A4_3} and \ref{tbl:A4_5} list the possibilities.
\begin{center}
\begin{minipage}{0.5\textwidth}
\begin{table}[H]
\begin{tabular}{ccccc}
$p \pmod{3}$ & $k$ & $\ell$ & $|x^{|G:M|}|$ & $|y^{|G:M|}|$\\
\hline
1  & 6 &  4 & 3 & 2\\
2  & 2 & 12 & 1 & 6
\end{tabular}
\caption{$M \cong A_4$, $p \equiv 3 \pmod{8}$}
\label{tbl:A4_3}
\end{table}
\end{minipage}%
\begin{minipage}{0.5\textwidth}
\begin{table}[H]
\begin{tabular}{ccccc}
$p \pmod{3}$ & $k$ & $\ell$ & $|x^{|G:M|}|$ & $|y^{|G:M|}|$\\
\hline
1  & 12 & 2 & 6 & 1\\
2  &  4 & 6 & 2 & 3
\end{tabular}
\caption{$M \cong A_4$, $p \equiv 5 \pmod{8}$}
\label{tbl:A4_5}
\end{table}
\end{minipage}
\end{center}
Now, $A_4$ contains no element of order 6, so $M$ is not \wexp{} with respect to $G$ when either $p \equiv 3 \pmod 8$ and $p \equiv 2 \pmod 3$, or $p \equiv 5 \pmod 8$ and $p \equiv 1 \pmod 3$. In the other two cases, one of $x^{|G:M|}$ or $y^{|G:M|}$ equals 2, and the other equals 3. By Lemma \ref{lem:PSLcyclic}(2), all elements of order 2 are conjugate in $G$, and likewise for elements of order 3. Since $A_4$ contains both elements of order 2 and elements of order 3, this means that some conjugate of $M$ contains $x^{|G:M|}$, and some conjugate of $M$ contains $y^{|G:M|}$.

We conclude that if $p \equiv 3 \pmod 8$, then $M$ is \wexp{} with respect to $G$ whenever $p \equiv \pm 3 \pmod{10}$ and $p \equiv 1 \pmod 3$. This is equivalent to having $p \equiv 43, 67 \pmod{120}$. Similarly, if $p \equiv 5 \pmod 8$, then $M$ has the desired property when $p \equiv 53, 77 \pmod{120}$.
\end{proof}

\begin{lem}
 \label{lem:whenS4good}
 Let $p \ge 13$ be prime.  Let $p \equiv \pm 3 \pmod {10}$ and $p \equiv \pm 1 \pmod 8$. Then, a maximal subgroup $M \cong S_4$ is \wexp{} with respect to $G = \PSL(2,p)$ if and only if $p \equiv \pm 7 \pmod {24}$.  In particular, if $p \not\equiv \pm 7 \pmod {24}$, then there exists an element $x \in G$ of order $(p-1)/2$ or $(p+1)/2$ such that the order of $x^{|G:M|}$ is $12$, and a maximal subgroup $M \cong S_4$ is \wexp{} with respect to $G$ if and only if
 \[ p \equiv 7, 17, 103, 113 \pmod {120}.\]
\end{lem}
\begin{proof}
Just as in the previous lemma, we consider cases depending on the residue of $p$ modulo 3 and use Lemmas \ref{lem:PSLcyclic} and \ref{lem:chainsaw}. The relevant data is shown in the tables below.
\begin{center}
\begin{minipage}{0.5\textwidth}
\begin{table}[H]
\begin{tabular}{ccccc}
$p \pmod{3}$ & $k$ & $\ell$ & $|x^{|G:M|}|$ & $|y^{|G:M|}|$\\
\hline
1  & 24 & 2 & 12 & 1\\
2  &  8 & 6 &  4 & 3
\end{tabular}
\caption{$M \cong S_4$, $p \equiv 1 \pmod{8}$}
\label{tbl:S4_1}
\end{table}
\end{minipage}%
\begin{minipage}{0.5\textwidth}
\begin{table}[H]
\begin{tabular}{ccccc}
$p \pmod{3}$ & $k$ & $\ell$ & $|x^{|G:M|}|$ & $|y^{|G:M|}|$\\
\hline
1  & 6 &  8 & 3 &  4\\
2  & 2 & 24 & 1 & 12
\end{tabular}
\caption{$M \cong S_4$, $p \equiv 7 \pmod{8}$}
\label{tbl:S4_7}
\end{table}
\end{minipage}
\end{center}
The group $S_4$ contains no element of order 12, but does contain elements of order 3 and elements of order 4. Moreover, in $G$, all elements of order 3 are conjugate, as are all elements of order 4. Arguing as in Lemma \ref{lem:whenA4good}, we conclude that $M$ is \wexp{} with respect to $G$ if and only if either $p \equiv 1 \pmod 8$ and $p \equiv 2 \mod 3$, or $p \equiv 7 \pmod 8$ and $p \equiv 1 \pmod 3$. Equivalently, $p \equiv 7, 17 \pmod{24}$.
\end{proof}

We are now ready to prove Theorem \ref{thm:PSLwic} in the case when $q$ is prime.

\begin{thm}
 \label{thm:PSLwic1}
 Let $p$ be a prime.  Then, $\PSL(2,p)$ is a \wic{} group if and only if $p = 2, 3, 5$ or 
  \[ p \equiv 7, 17, 43, 53, 67, 77, 103, 113 \pmod {120}.\]
\end{thm}

\begin{proof}
% First, when $p = 5$, $\PSL(2,5) \cong A_5$ and is weakly index closed by Lemma \ref{lem:A5}.  
 
First, when $p = 2$, $3$, or $5$, $\PSL(2,p)$ is isomorphic to (respectively) $S_3$, $A_4$, or $A_5$, each of which is \wic{} by Theorem \ref{thm:solvwic} or Lemma \ref{lem:A5}.  

 The proof that $\PSL(2,7)$ is \wic{} is similar to that for $A_5$. The group $\PSL(2,7)$ contains two classes of maximal subgroups isomorphic to $S_4$ (index $7$) and one conjugacy class of subgroups isomorphic to $C_7 \rtimes C_3$ (index $8$).  There is a single conjugacy class of elements of order $2$, order $3$, and order $4$, and, while there are two conjugacy classes of elements of order $7$, elements from both conjugacy classes are contained in a single $C_7 \rtimes C_3$ maximal subgroup.  Again, the result follows from inspection.
 
 Next, $\PSL(2,11)$ is not \wic: let $x$ be an element of order $6$ and let $M$ be a maximal subgroup isomorphic to $A_5$ (index $11$).  Then, $x^{|G:M|}$ has order $6$, but $A_5$ does not contain elements of order $6$.
 
 Finally, let $p \ge 13$ be prime.  By Proposition \ref{prop:PSLwicA4S4A5}, it suffices to check whether maximal subgroups isomorphic to one of $A_4$, $A_5$, or $S_4$ are \wexp.  By Lemmas \ref{lem:A5inPSLbad}, \ref{lem:whenA4good}, \ref{lem:whenS4good}, all such maximal subgroups are \wexp{} when $p \equiv \pm 3 \pmod {10}$ and when $p \equiv \pm 5, \pm 7 \pmod {24}$, i.e., exactly when 
 \[ p \equiv 7, 17, 43, 53, 67, 77, 103, 113 \pmod {120},\]
 as desired.
\end{proof}

\begin{rem}
 \label{rem:nicex}
 Lemmas \ref{lem:A5inPSLbad}, \ref{lem:whenA4good}, \ref{lem:whenS4good}, and the proof of Theorem \ref{thm:PSLwic1} show that, if $G = \PSL(2,p)$ is \indexc{}, then there exists $x \in G$ of order $(p \pm 1)/2$ and $H \le G$ such that $H$ does not contain an element of order $|x^{|G:H|}|$. 
\end{rem}

\subsection{The groups $\PSL(2,q)$ that are \wic}

This subsection is dedicated to the proof of Theorem \ref{thm:PSLwic}.  Throughout this subsection, $q$ will be a prime power.

\begin{prop}
 \label{prop:q02}
 Let $q = q_0^2$, where $q_0 \ge 4$.  Then, $G = \PSL(2,q)$ is not \wic.
\end{prop}

\begin{proof}
By Theorem \ref{thm:maxPSL}, $G$ contains a maximal subgroup $M$ isomorphic to $\PGL(2,q_0)$.  If $e = \gcd(q-1,2)$, then $|G:M| = q_0(q + 1)/e$.
 
 Let $x$ be an element of order $(q-1)/e$ in $G$.  Since $\gcd(|G:M|, (q-1)/e) = 1$, the element $x^{|G:M|}$ has order $(q-1)/e$.  Since the maximum order of an element in $M \cong \PGL(2,q_0)$ is $q_0 + 1$ and 
 \[ q_0 + 1 < \frac{q_0^2 - 1}{e} = \frac{q - 1}{e}, \]
 we see that $x^{|G:M|}$ is not in any conjugate of $M$ in $G$.  Thus, $M$ is not \wexp, and hence $G$ is not \wic.
\end{proof}

\begin{lem}
 \label{lem:powerupPSLq}
 %Let $q = q_0^r$, where $r$ is an odd prime. 
Let $q = q_0^n$, where $n$ is an odd integer. 
 Let $G = \PSL(2,q)$, $H$ be a subgroup of $G$ isomorphic to $\PSL(2,q_0)$, and $e = \gcd(q-1, 2)$.
 \begin{enumerate}[(1)]
  \item $|G:H| = k \cdot \ell \cdot m$, where $k = (q-1)/(q_0 - 1)$, $\ell = q/q_0$, $m = (q+1)/(q_0 + 1)$,  and $k$, $\ell$, and $m$ are pairwise coprime.
  \item If $x$ is an element of order $p$ in $G$, then $x^{|G:H|} = 1$.
  \item If $x$ is an element of order $(q - 1)/e$ in $G$, then $x^{|G:H|}$ has order $(q_0 - 1)/e$ and is contained in a conjugate of $H$ in $G$.
  \item If $x$ is an element of order $(q + 1)/e$ in $G$, then $x^{|G:H|}$ has order $(q_0 + 1)/e$ and is contained in a conjugate of $H$ in $G$.
  \item $H$ is \wexp{} with respect to $G$.
 \end{enumerate}
\end{lem}

\begin{proof}
 First, since $q = q_0^n$, $n$ odd, we have $|G:H| = k \cdot \ell \cdot m$ as in the statement of (1), where $k$, $\ell$, and $m$ are integers.  It is immediate that $\gcd(k, \ell) = \gcd(\ell, m) = 1$.  Since $\gcd(q - 1, q + 1) = e \le 2$ and $q_0 \pm 1$ is even if and only if $e = 2$, we see that $\gcd(k,m) = 1$ as well, proving (1).
 
 For (2), if $x$ is an element of order $p$, then $p$ divides $\ell$, and so
 \[ x^{|G:H|} = x^\ell = 1.\]
 
 For (3), if $x$ is an element of order $(q-1)/e$, then, by (1), we see that $|x^{|G:H|}| = |x^k|$, and so $x^{|G:H|}$ has order
 \[ \frac{\phantom{x}\frac{q-1}{e}\phantom{x}}{\frac{q-1}{q_0 - 1}} = \frac{q_0 - 1}{e}.\]
 By Lemma \ref{lem:PSLcyclic}(1), all cyclic subgroups of order $(q_0 - 1)/e$ are conjugate in $G$.  Since $H$ contains a cyclic subgroup of order $(q_0 - 1)/e$, we see that $x^{|G:H|}$ is contained in some conjugate of $H$, as desired. 
 
 The proof for (4) is analogous to the proof of (3).

To prove (5), by Lemma \ref{lem:PSLcyclic}(1), it suffices to check elements of order $(q - 1)/e$, $p$, and $(q + 1)/e$ in $G$.  If $x$ is an element of one of these orders, then, by the previous parts of the lemma, $x^{|G:H|}$ is contained in a conjugate of $H$, showing that $H \lwexp G$.
\end{proof}

\begin{comment}
\begin{lem}
 \label{lem:q0n}
 Let $q = q_0^n$, where $n$ is odd, and let $G = \PSL(2,q)$.  If $H \cong \PSL(2,q_0)$ is a maximal subgroup of $G$ and is weakly index closed, then $H$ is \wexp{} with respect to $G$.
\end{lem}

\begin{proof}
 Let $q$, $q_0$, and $G$ be as in the statement, and let $e = \gcd(q-1,2)$.  Assume $H \cong \PSL(2,q_0)$ is \wic.  By Lemma \ref{lem:PSLcyclic}(1), it suffices to check elements of order $(q - 1)/e$, $p$, and $(q + 1)/e$ in $G$.  If $x$ is an element of one of these orders, then, by Lemma \ref{lem:powerupPSLq}, $x^{|G:H|}$ is contained in a conjugate of $H$, showing that $H$ is \wexp{} with respect to $G$.
\end{proof}
\end{comment}

\begin{prop}
\label{prop:qdependsonp}
 Let $q = p^d$, where $d$ is odd, and let $G = \PSL(2,q)$.  Then, $G$ is \wic{} if and only if $\PSL(2,p)$ is \wic.
\end{prop}

\begin{proof}
 Assume $\PSL(2,p)$ is \wic.  We proceed by induction on the number of (not necessarily distinct) primes in a factorization of $d$.  First, let $d$ be an odd prime.  Then, by Theorem \ref{thm:maxPSL}, and Proposition \ref{prop:PSLmaxgood}, it suffices to consider maximal subgroups of $G = \PSL(2,p^d)$ isomorphic to $\PSL(2,p)$, which by Lemma \ref{lem:powerupPSLq}(5) are \wexp{} with respect to $G$.  Thus, by Lemma \ref{lem:checkmax}, $\PSL(2,q)$ is \wic{} when $q = p^d$, $d$ an odd prime.  
 
 Now, assume the result is true when $d$ is a product of $n - 1$ primes, where $n \ge 2$.  So, assume $d$ is the product of $n$ primes.  Again, by Theorem \ref{thm:maxPSL}, and Proposition \ref{prop:PSLmaxgood}, it suffices to consider maximal subgroups of $G = \PSL(2,p^d)$ isomorphic to $\PSL(2, q_0)$, where $q = q_0^r$, $r$ prime.  Since $q_0 = p^{d/r}$ and $d/r$ is the product of $n - 1$ primes, $\PSL(2, q_0)$ is \wic. Moreover, $\PSL(2,q_0)$ is \wexp{} with respect to $G$ by Lemma \ref{lem:powerupPSLq}(5).  Hence, $G$ is itself \wic.  Therefore, if $d$ is odd, $\PSL(2, p^d)$ is \wic{} if $\PSL(2,p)$ is. 
 
Conversely, suppose that $p$ is such that $\PSL(2,p)$ is not \wic, and let $H \le G$ be isomorphic to $\PSL(2,p)$.  By Remark \ref{rem:nicex}, we may assume that there exists $h \in H$ of order either $(p - 1)/2$ or $(p + 1)/2$ and $L \le H$ such that $L$ does not contain an element of order $|h^{|H:L|}|$.  By Lemma \ref{lem:PSLcyclic}, $h$ is contained in $\langle x \rangle$, where the order of $x$ is $(q - 1)/2$ or $(q + 1)/2$, respectively.  By Lemma \ref{lem:powerupPSLq}(3),(4), $x^{|G:H|}$ has order $(p - 1)/2$ or $(p+1)/2$, respectively, and hence $\langle x^{|G:H|} \rangle = \langle h \rangle$. Consequently, $h = (x^a)^{|G:H|}$ for some integer $a$.  This means that
 \[ (x^a)^{|G:L|} = ( (x^a)^{|G:H|})^{|H:L|} = h^{|H:L|}, \] 
 and, since $L$ does not contain any elements of this order, $(x^a)^{|G:L|}$ is not in any conjugate of $L$ in $G$.  Therefore, $L$ is not \wexp{} with respect to $G$, and $G$ is not \wic, completing the proof.
\end{proof}

We now have enough results to prove Theorem \ref{thm:PSLwic}.

\begin{proof}[Proof of Theorem \ref{thm:PSLwic}]
 Let $p$ be a prime, $q = p^d$, and $G = \PSL(2,q)$.  If $d = 1$, then the result follows from Theorem \ref{thm:PSLwic1}, so we assume $d \ge 2$.  Suppose first that $d$ is even. If $p = 2$ and $d = 2$, then $G \cong A_5$ and is \wic{} by Lemma \ref{lem:A5}, and, if $p = 3$ and $d = 2$, then $G \cong A_6$ and is \wic{} by Lemma \ref{lem:A6}. Otherwise, $q \ge 16$ and $G$ is not \wic{} by Proposition \ref{prop:q02}.    Finally, if $d$ is odd, then $G$ is \wic{} if and only if $\PSL(2,p)$ is \wic{} by Proposition \ref{prop:qdependsonp}, completing the proof.
\end{proof}

\bibliographystyle{plainurl}
\bibliography{Index}
 
\end{document}